\documentclass[10pt,reqno]{amsart}
\usepackage[T1]{fontenc}
\usepackage{lmodern}
\usepackage{amsmath,amssymb,mathtools}
\usepackage{microtype}
\usepackage{xcolor}
\definecolor{revisiongreen}{rgb}{0,0.42,0.12}

\newcommand{\revisionb}[1]{{\color{revisiongreen}\hypersetup{linkcolor=revisiongreen,citecolor=revisiongreen,urlcolor=revisiongreen}#1}}
\usepackage[colorlinks]{hyperref}
\makeatletter
\@namedef{subjclassname@2020}{\textup{2020} Mathematics Subject Classification}
\makeatother
\hypersetup{pdftitle={Algebraic ellipticity of quartic double spaces},pdfauthor={Alexander Dvorsky, Shulim Kaliman, and Mikhail Zaidenberg}}

\newtheorem{theorem}{Theorem}[section]
\newtheorem{proposition}[theorem]{Proposition}
\newtheorem{lemma}[theorem]{Lemma}
\newtheorem{corollary}[theorem]{Corollary}
\theoremstyle{remark}
\newtheorem{remark}[theorem]{Remark}

\newcommand{\A}{\mathbb A}
\newcommand{\PP}{\mathbb P}
\newcommand{\OO}{\mathcal O}
\newcommand{\EE}{\mathcal E}
\newcommand{\Pic}{\operatorname{Pic}}
\newcommand{\Gr}{\operatorname{Gr}}
\newcommand{\Gm}{\mathbb G_m}
\newcommand{\Ga}{\mathbb G_a}
\DeclareMathOperator{\pr}{pr}
\DeclareMathOperator{\id}{id}
\newcommand{\doi}[1]{\href{https://doi.org/#1}{\nolinkurl{#1}}}

\title[Ellipticity of quartic double spaces]{Algebraic ellipticity 
of quartic double spaces}
\author{Alexander Dvorsky, Shulim Kaliman, and Mikhail Zaidenberg}
\dedicatory{In memory of YURI G. ZARHIN}
\address{Department of Mathematics, University of Miami, 
Coral Gables, FL 33124, USA}
\email{dvorsky@gmail.com}
\address{Department of Mathematics, University of Miami, 
Coral Gables, FL 33124, USA}
\email{kaliman@math.miami.edu}
\address{Univ. Grenoble Alpes, CNRS, IF, 
38000 Grenoble, France}
\email{mikhail.zaidenberg@gmail.com}
\date{}
\subjclass[2020]{Primary 14J30, 14M20; Secondary 14E07, 32Q56}
\keywords{Algebraic ellipticity, quartic double space, 
cubic hypersurface, birational involution, unirationality, algebraic spray}

\begin{document}

\begin{abstract}
We prove that every smooth quartic double $n$-fold, 
$n\ge2$, over an algebraically closed field of characteristic zero 
is algebraically elliptic in Gromov's sense. In particular, 
every smooth quartic double solid is algebraically elliptic, 
which answers Question~4.28 in \cite{Zai}.
We also prove that removing any closed subset of codimension at least two
preserves algebraic ellipticity for smooth quartic double spaces and
smooth cubic hypersurfaces of dimension at least two.
\end{abstract}

\maketitle

\section{Introduction}
Throughout, all varieties and morphisms are 
algebraic and defined
over an algebraically closed field $k$ of characteristic zero.
An \emph{algebraic spray} on a smooth algebraic variety $X$ is 
a vector bundle $\rho:E\to X$ together with a morphism 
$s:E\to X$ whose restriction to the zero section is $\rho$. 
It is \emph{dominating} if, for every $x\in X$, 
the differential in the fiber direction
\[
 d(s|_{E_x})_{0_x}:T_{0_x}E_x\longrightarrow T_xX
\]
is surjective. The variety $X$ is \emph{algebraically elliptic
in Gromov's sense}, or simply  \emph{algebraically elliptic},
if it admits a dominating algebraic spray, see \cite{Gro}. 

Every smooth complete rational variety is algebraically elliptic, 
by the uniform retract rationality theorem of Banecki \cite{Ban}
and its application to complete varieties.
For $n\ge 2$, every smooth cubic hypersurface 
in $\PP^{n+1}$ is algebraically elliptic, 
see \cite{KZcubic}. Thus, smooth cubic threefolds and 
very general cubic fourfolds 
give the first examples of 
algebraically elliptic smooth projective varieties which 
are irrational; the irrationality is due to Clemens and 
Griffiths \cite{CG} for cubic threefolds and to Katzarkov, 
Kontsevich, Pantev and Yu
\cite{KKPY} for cubic fourfolds.

In this paper, we show that similar results hold also 
for quartic double spaces. Let, for $n\ge2$,
\[
 \pi:X\longrightarrow\PP^n
\]
be a double cover branched along a smooth quartic 
hypersurface in $\PP^n$. It can be given in 
the weighted projective space $\PP(1^{n+1},2)$ 
by equation
\[
 w^2=F(x_0,\ldots,x_n),
\]
where $F$ is a smooth quartic form. We call $X$ 
a \emph{smooth quartic double $n$-fold}.

\begin{theorem}\label{thm:main} For $n\ge 2$,
every smooth quartic double $n$-fold over $k$ 
is algebraically elliptic.
\end{theorem}

For $n=3$, this answers Question~4.28 in \cite{Zai}, 
which asks whether any smooth quartic double solid 
is algebraically elliptic. Ciliberto and the third author 
had previously constructed in \cite{CZ} small algebraic 
resolutions of certain \emph{nodal} rational quartic double 
solids that are algebraically elliptic.

This theorem supplies the first examples of smooth 
projective elliptic varieties 
that are not stably rational. Indeed, it is unknown 
whether general smooth cubic threefolds 
or fourfolds are stably irrational. Recently, it has been proved that  
$X\times\PP^1$ is irrational for 
any smooth cubic threefold $X$ in $\PP^4$, see \cite{GX}.

Artin and Mumford \cite{AM} constructed a singular 
quartic double solid with a smooth resolution that is 
unirational and not stably rational. The stable-irrationality 
of very general quartic double solids and fourfolds
is due, respectively,
to Voisin \cite{Voi} and to Hassett, Pirutka and Tschinkel  \cite{HPT}; 
a proof of their unirationality can be found, e.g., in Beauville \cite{Bea}.
Based on these results, we deduce the following corollary of 
Theorem \ref{thm:main}.

\begin{corollary}\label{cor:notstable}
For $n=3,4$, a very general smooth quartic double 
$n$-fold over $\mathbb C$ is algebraically elliptic 
and not stably rational.
\end{corollary}

The next corollary is analogous to Corollary 1.4 in \cite{KZcubic}.

\begin{corollary}\label{cor:cone}
Let $X$
be a smooth quartic double 
$n$-fold, $n\ge 2$, and $Y$ be a 
punctured affine cone over $X$ 
defined by an ample line bundle $L$, namely
\[
 Y=\operatorname{Spec}\!\left(\bigoplus_{m\ge0}H^0(X,L^{\otimes m})\right)
       \setminus\{\text{vertex}\}\simeq (L^{-1})^\times,
\]
where $\times$ denotes removal of the zero section. 
Then the following hold.
\begin{itemize}
\item[(a)]
$Y$ is elliptic in Gromov's sense.  
\item[(b)]
There exist surjective morphisms $\A^{n+1}\to X$ resp.
$\A^{n+2}\to Y$ which are smooth and surjective
on appropriate open subsets $U\subset\A^{n+1}$ 
resp. $V\subset\A^{n+2}$.
\item[(c)] If $k=\mathbb C$, then there exists 
a surjective morphism $\A^n\to X$
 which is smooth and surjective
on an appropriate open subset $U\subset\A^n$.
\item[(d)] The monoid of endomorphisms ${\rm End}(Y)$ 
acts $m$-transitively on $Y$ for every natural number $m$. 
\end{itemize}

\end{corollary}

The same construction gives the following removal theorem, 
which was motivated by the main result of \cite{FKZ16}.

\begin{theorem}\label{thm:removal}
Let $X$ be either a smooth quartic double $n$-fold or a smooth cubic
hypersurface in $\PP^{n+1}$, where $n\ge2$. If $Z\subset X$ is a closed
subset of codimension at least two, then $X\setminus Z$ is algebraically
elliptic.
\end{theorem}

The construction of a dominating spray on a cubic hypersurface
in \cite{KZcubic} 
exploits the technique of a moving family of involutions. 
Suppose that a smooth complete unirational $n$-fold $X$ 
carries a rational map 
\[X\times X \dasharrow X,\qquad
 (p,x)\longmapsto i_p(x),
\]
such that
\begin{itemize} \item
for a fixed $p\in X$, $i_p\colon X \dasharrow X$ is 
an involution in $x$ defined on a dense open subset of $X$, and
\item for a fixed $x\in X$, the rational map 
$X \dasharrow X$, $p\longmapsto i_p(x)$,
is dominant. 
\end{itemize} Following \cite{KZcubic}, we show in 
Proposition~\ref{prop:criterion} that the existence 
of such a moving family  of birational involutions
on a smooth complete unirational variety
implies its algebraic ellipticity. 

For a smooth cubic hypersurface $X$ in $\PP^{n+1}$
such a family of involutions  is classically known; 
indeed, let $i_p(x)$ be the third point 
of intersection of $X$ with the line through $p$ and $x$. 
For fixed $x$, 
symmetry gives $i_p(x)=i_x(p)$, where $i_x$ is the 
birational involution associated with the projection from $x$.
In particular, $p\mapsto i_p(x)=i_x(p)$ is dominant.

For a quartic double space $\pi\colon X\to \PP^n$, 
such a moving family of involutions 
can be constructed as follows.
Let $\ell$ be  a general line in $\PP^n$.
The inverse image $E_\ell=\pi^{-1}(\ell)$ on $X$ 
is a smooth elliptic curve. Let 
\[\sigma\colon X\to X,\qquad 
(x_0:\cdots:x_n:w)\longmapsto (x_0:\cdots:x_n:-w)\]
be the deck involution. Using an addition on $E_\ell$, 
 we define a hyperelliptic involution
\[
 i_p(x)=2\sigma(p)-x,\qquad p,x\in E_\ell,
\]
with ``center'' $p$. 
This formula is independent of the choice of origin. 
Taking $\sigma(p)$ as origin on $E_\ell$, 
the involution becomes $x\mapsto -x$ 
and its four fixed points are the two-torsion points.
For a fixed $x\in X$, we consider 
the family of elliptic curves $E_\ell=\pi^{-1}(\ell)$ on $X$ 
passing through $x$, 
above lines $\ell$ in $\PP^n$ passing through $\pi(x)$. 
Varying $p\in X$, the moving family  of involutions $\{i_p\}$
has the desired properties. Indeed,
the rational map $X \dasharrow X$, $p\longmapsto i_p(x)$,
is dominant of degree four. 
This remains true when $x$ lies on the ramification divisor.

To prove Theorem~\ref{thm:removal}, we use the same moving families and
the fact that the dimension of the inverse image of every point $y \in X$
under the reflection map is at most $n$.               
This allows the complete
parameter lines used in the spray construction to avoid $Z$.

The contents of the paper are as follows. 
In Section \ref{sec:criterion}, 
we provide the ellipticity criterion in terms of moving 
families of involutions, a removal criterion, and the fiber estimate for
cubic hypersurfaces. In Section \ref{sec:double},
we construct 
the required families of involutions 
for quartic double spaces and prove their fiber estimate, completing the
proofs of Theorems~\ref{thm:main} and~\ref{thm:removal}.
Theorem~\ref{thm:generalcone} in Section \ref{sec:cones} 
provides a general criterion
of ellipticity of affine cones in terms of moving families 
of involutions, that contains,  as special cases, the ellipticity 
of affine cones over cubic hypersurfaces 
and over double quartic spaces. 
In particular, this  leads to a proof of Corollary \ref{cor:cone}. 

\section{A moving-involution criterion}\label{sec:criterion}

In this section, we present a modified version of the ellipticity 
criterion in presence of a moving  family of involutions, which is used in 
\cite[Section 4]{KZcubic}.

 We first recall the local form of the extension and localization 
theory for algebraic sprays. A \emph{local rank-one spray 
with values in $X$} is a morphism
\[
 s:U\times\A^1\longrightarrow X,\qquad s(z,0)=z,
\]
where $U\subset X$ is open. The image need not lie in $U$. 
Its parameter derivative at $z\in U$ is $d(s|_{\{z\}\times\A^1})_0$.

\begin{lemma}[{\rm  \cite[Lemma~2.5]{KZcubic}}]\label{lem:localization}
Let $X$ be a smooth algebraic variety of dimension $n$.
Suppose that for every $x\in X$ there are local rank-one sprays 
with values in $X$, defined in an open neighborhood of $x$, 
whose parameter derivatives 
span $T_xX$. Then $X$ is algebraically elliptic.
\end{lemma}

The extension and localization results used here apply to smooth varieties
without a quasiprojectivity assumption; see \cite{KZcone}, together with
\cite[Theorem~1.1]{KZsub} for the equivalence of ellipticity and subellipticity.

The following lemma is analogous to Lemma 4.1 
and Corollary 4.2 in \cite{KZcubic}.

\begin{lemma}\label{lem:avoid}
Let $Y$ be a variety, let $n\ge2$, and let 
$I\subset Y\times\PP^n$ be closed. 
For $y\in Y$, let 
\[
 \quad I_y=\{a'\in\PP^n:(y,a')\in I\}.
\]
 Suppose that $\dim I_y\le n-2$ for some $y\in Y$.
Then for any $a\in\PP^n\setminus I_y$ 
there are $n$ projective lines $L_1,\ldots,L_n$ 
in $\PP^n$ through $a$, 
with linearly independent tangent directions at $a$, 
and an open neighborhood $W$ of $y$ in $Y$ such that
\[
 (W\times L_i)\cap I=\varnothing\qquad(1\le i\le n).
\]
\end{lemma}

\begin{proof} Let $C_a=C_a(I_y)\subset\PP^n$ 
be the projective cone 
with base $I_y$ and with vertex $a$, and let 
$C_a'\subset T_a\PP^n$ be  
its tangent cone at the vertex.  
The projectivization $\PP C_a'$ is 
a closed subset of $\PP (T_a\PP^n)\simeq\PP^{n-1}$ 
of dimension at most $n-2$. 
Its complement is a nonempty 
open subset of $\PP (T_a\PP^n)$. 
Choose $n$ linearly independent 
tangent directions in this complement. 
The projective lines $L_i$ in $\PP^n$ 
through $a$, tangent to chosen directions, 
are not among the ruling lines of the cone $C_a$, 
and therefore miss $I_y$. 
The image of 
$I\cap(Y\times L_i)$ in $Y$ 
is closed because $L_i$ is projective, 
and it does not contain $y$. The complement 
of the union of these images 
is the required neighborhood $W$.
\end{proof}
The next proposition provides a criterion of ellipticity in terms 
of moving families of involutions. It generalizes the approach 
of Proposition 4.3 
and Corollary 4.4 in \cite{KZcubic}\footnote{The reference 
to Kusakabe's result in the proof of Proposition 4.3 
in \cite{KZcubic} should be replaced by the one to 
its generalization in \cite[Theorem 1.7]{AKZ}.}.

\begin{proposition}
\label{prop:criterion}
Let $X$ be a smooth complete unirational variety 
of dimension $n\ge2$. Suppose that a rational map
\[
 i:X\times X\dashrightarrow X,\qquad (p,x)\longmapsto i_p(x),
\]
is regular on an open subset $\Omega\subset X\times X$ such that
\begin{equation}\label{eq:invol}
 (p,i_p(x))\in\Omega,\qquad i_p(i_p(x))=x
 \qquad\text{for every }(p,x)\in\Omega.
\end{equation}
Assume that for every $x\in X$ the open subset
\[
 P_x=\{p\in X:(p,x)\in\Omega\}
\]
is dense and the evaluation map
\[
 e_x:P_x\longrightarrow X,\qquad p\longmapsto i_p(x)
\]
is dominant. Then $X$ is algebraically elliptic.
\end{proposition}

\begin{proof}
Choose a dominant rational map $\varphi:\PP^n\dashrightarrow X$. 
Consider the rational map
\begin{equation}\label{eq:Psi}
 \Psi:X\times\PP^n\dashrightarrow X,
 \qquad \Psi(x,a)=i_{\varphi(a)}(x),
\end{equation}
and let $I$ be the indeterminacy locus of $\Psi$. Since the target $X$
is proper and the source $X\times\PP^n$ is smooth, the valuative
criterion for properness shows that $\Psi$ is defined at every
codimension-one point of its source. Thus
\[
 \operatorname{codim}_{X\times\PP^n}I\ge2.
\]
Since $\dim(X\times\PP^n)=2n$, every component of $I$ 
dominating $X$ has generic fiber of dimension at most $n-2$. 
Components not dominating $X$ have proper closed image under 
the proper projection $I\to X$. Removing from $X$ those images 
and the upper-semicontinuity jumping loci of dominating components of 
$I$ gives a nonempty open subset $V\subset X$ such that
\begin{equation}\label{eq:fiberbound}
 \dim I_y\le n-2\qquad \forall y\in V.
\end{equation}

Fix an arbitrary $x\in X$. Let $D_x\subset\PP^n$ be the open subset 
where $\varphi:\PP^n\dashrightarrow X$ is regular and 
$\varphi(a)\in P_x$ for any $a\in D_x$. The morphism
\[
 F_x:D_x\longrightarrow X,
 \qquad F_x(a)=i_{\varphi(a)}(x)
\]
is dominant. Its source and target are smooth 
of the same dimension. So (in characteristic zero) $F_x$ is \'etale on 
a dense open subset. 
We may therefore choose $a\in D_x$ such that $dF_x|_a$ 
has rank $n$ and
\[
y:=F_x(a)\in V.
\]
Letting $p=\varphi(a)\in X$ we obtain $y=i_p(x)$ and $x=i_p(y)$, 
where $(p,x)$ and $(p,y)$ both belong to $\Omega$ 
by our assumptions in \eqref{eq:invol}.
Thus, for every $x\in X$ we have found a pair $(p,y)\in\Omega$
 such that $x=i_p(y)$.
 If
\[
 \Omega_p=\{z\in X:(p,z)\in\Omega\},
\]
then $i_p:\Omega_p\to\Omega_p$ is an involution, 
hence an isomorphism. In particular, $i_p$ identifies 
neighborhoods of $x$ and $y$, and $\Psi$ is regular at $(y,a)$, 
see \eqref{eq:Psi}.

For $a'$ near $a$, the involution identity gives
\[
 \Psi\bigl(F_x(a'),a'\bigr)=x.
\]
Differentiating at $a$ yields
\[
 D_1\Psi_{(y,a)}\circ dF_x|_a+D_2\Psi_{(y,a)}=0.
\]
Both $D_1\Psi_{(y,a)}=di_p|_y$ and $dF_x|_a$ have rank $n$. 
Consequently
\begin{equation}\label{eq:rank}
 D_2\Psi_{(y,a)}:T_a\PP^n\longrightarrow T_xX
 \quad\text{is an isomorphism.}
\end{equation}

We have $a\notin I_y$ and the bound \eqref{eq:fiberbound}. 
Lemma~\ref{lem:avoid} gives projective lines $L_1,\ldots,L_n$ 
through $a$ with independent tangent directions and 
a neighborhood $W$ of $y$ such that $\Psi$ is regular on 
$W\times L_i$. Choose an open neighborhood $U$ 
of $x$ inside $\Omega_p$ so that $i_p(U)\subset W$. 
For each $i$, choose an isomorphism
\[
 \lambda_i:\A^1\xrightarrow{\sim}L_i\setminus\{q_i\},
 \qquad \lambda_i(0)=a,
\]
where $q_i\ne a$. Define
\begin{equation}\label{eq:spray}
 s_i:U\times\A^1\longrightarrow X,
 \qquad s_i(z,t)=\Psi\bigl(i_p(z),\lambda_i(t)\bigr).
\end{equation}
Each $s_i$ is a morphism and
\[
 s_i(z,0)=i_p(i_p(z))=z.
\]
The parameter derivative of $s_i$ at $(x,0)$ is $D_2\Psi_{(y,a)}$ 
applied to the tangent vector of $L_i$. 
These derivatives span $T_xX$ by \eqref{eq:rank}. 
Thus, \eqref{eq:spray} provides 
$n$ rank-one sprays
defined on an open neighborhood $U$ of $x$ in $X$ 
such that their parameter derivatives at $x$ span $T_xX$. 
Since $x\in X$ is arbitrary, Lemma~\ref{lem:localization} applies 
and gives the result. 
\end{proof}

For a rational map $i:X\times X\dashrightarrow X$, let $R$ denote its
maximal open subset of regularity. In the following statements,
$i^{-1}(Z)$ means the inverse image in $R$; taking its closure will always
be indicated explicitly.

\begin{proposition}\label{prop:removal}
Let $X$ and $i$ satisfy the hypotheses of
Proposition~\ref{prop:criterion}, and let $Z\subset X$ be closed of
codimension at least two. Suppose, in addition, that
\begin{enumerate}
\item[(a)] for every $p\in X$, the open subset
$R_p=\{z\in X:(p,z)\in R\}$ is dense and the morphism
$i_p:R_p\to X$ is dominant;
\item[(b)] $\dim i^{-1}(Z)\le2n-2$.
\end{enumerate}
Then $X\setminus Z$ is algebraically elliptic.
\end{proposition}

\begin{proof}
Choose a dominant rational map $\varphi:\PP^n\dashrightarrow X$ and
use the map $\Psi$ of \eqref{eq:Psi}. The map $\varphi$ is generically
finite, since its source and target have the same dimension.
Let $I$ be the indeterminacy locus of $\Psi$
and set
\begin{equation}\label{eq:removedset}
 J=I\cup\overline{\Psi^{-1}(Z)}\subset X\times\PP^n,
\end{equation}
where the inverse image is taken on the regularity locus of $\Psi$.
We first prove that $\dim J\le2n-2$.

As in the proof of Proposition~\ref{prop:criterion}, $\dim I\le2n-2$.
Suppose that a prime divisor $D\subset X\times\PP^n$ is contained in
$\overline{\Psi^{-1}(Z)}$. Since $Z$ is closed, $\Psi$ takes values in
$Z$ at every point of $D$ where it is regular.

Let $\pr_2:X\times\PP^n\to\PP^n$ be the second projection.
Suppose first that $\pr_2(D)=\PP^n$. Choose a dense open subset
$A\subset\PP^n$ on which $\varphi$ is a quasi-finite morphism, and
consider
\[
 \nu_A:X\times A\longrightarrow X\times X,
 \qquad \nu_A(z,a)=(\varphi(a),z).
\]
The fiber of $\nu_A$ over $(p,z)$ consists of the points $(z,a)$ with
$\varphi(a)=p$, so it is finite. The dimension theorem therefore gives
$\dim\nu_A^{-1}(W)\le\dim W$ for every locally closed subset
$W\subset X\times X$. In particular, hypothesis~(b) gives
\[
 \dim\nu_A^{-1}\bigl(i^{-1}(Z)\bigr)
 \le\dim i^{-1}(Z)\le2n-2.
\]
Since the indeterminacy locus $S$ of $i$ 
has dimension at most $2n-2$ and $\nu_A$ is quasi-finte,
one has $\dim {\nu_A^{-1} (S)} \leq 2n-2$. 
 Thus, the generic point $\omega$ 
of $D$ is not contained in ${\nu_A^{-1} (S)}$ and,
consequently, $i$ is regular at $\nu_A(\omega)$. 
In particular, $\omega$ is contained in 
$\Psi^{-1}(Z) =\nu_A^{-1}\bigl(i^{-1}(Z)\bigr)$
contrary to the fact that $\dim D=2n-1$.

Suppose next that $D$ does not dominate $\PP^n$ under
$\pr_2$, and put $E=\pr_2(D)\subset\PP^n$. The projection $\pr_2$
is proper because $X$ is complete, so $E$ is closed. Each fiber of
$\pr_2|_D:D\to E$ is a closed subset of $X\times\{a\}$ and has
dimension at most $n$. Hence
\[
 2n-1=\dim D\le\dim E+n.
\]
Since $E$ is a proper closed subset of $\PP^n$, it follows that
$\dim E=n-1$. Thus $E$ is a divisor in $\PP^n$. The inclusion
$D\subset X\times E$ is an equality, since both varieties are
irreducible and have dimension $2n-1$.

The map $\varphi$ is regular at the generic point of $E$, because
its target $X$ is proper. Choose $a\in E$ where $\varphi$ is regular
and put $p=\varphi(a)$. By hypothesis~(a), $R_p$ is dense in $X$ and
$i_p:R_p\to X$ is dominant. For every $z\in R_p$, the composite
$\Psi$ is regular at $(z,a)$, since $\varphi$ is regular at $a$ and
$i$ is regular at $(p,z)$. Since $(z,a)\in D$ and
$D\subset\overline{\Psi^{-1}(Z)}$, we obtain
\[
 i_p(z)=\Psi(z,a)\in Z\qquad(z\in R_p).
\]
Thus $i_p(R_p)\subset Z$, contradicting the dominance required in
(a). This proves the bound $\dim J\le2n-2$.

The proper projection $J\to X$ now gives a dense open subset
$V\subset X$ such that $\dim J_y\le n-2$ for $y\in V$.
Fix $x\in X\setminus Z$. As in the proof of
Proposition~\ref{prop:criterion}, choose $a\in\PP^n$ such that
\[
 p=\varphi(a),\qquad y=F_x(a)=i_p(x)\in V,
 \qquad (p,x)\in\Omega,
\]
and $dF_x|_a$ is an isomorphism. Then $(p,y)\in\Omega$ and
$\Psi(y,a)=x$. Thus $(y,a)\notin J$: the map $\Psi$ is regular there
and its value lies outside the closed subset $Z$.

Apply Lemma~\ref{lem:avoid} to $J$. It gives projective lines
$L_1,\ldots,L_n$ through $a$ with independent tangent directions,
and a neighborhood $W$ of $y$ such that
\[
 \Psi:W\times L_i\longrightarrow X\setminus Z
\]
is a morphism for every $i$. Choose a neighborhood
$U\subset X\setminus Z$ of $x$ contained in $\Omega_p$ with
$i_p(U)\subset W$. With the parametrizations $\lambda_i$ used in
\eqref{eq:spray}, define
\[
 s_i:U\times\A^1\longrightarrow X\setminus Z,
 \qquad s_i(z,t)=\Psi\bigl(i_p(z),\lambda_i(t)\bigr).
\]
These are local rank-one sprays. The identity
$\Psi(F_x(a'),a')=x$ gives
\[
 D_2\Psi_{(y,a)}=-di_p|_y\circ dF_x|_a,
\]
so their parameter derivatives span $T_xX=T_x(X\setminus Z)$.
Since $x$ is arbitrary, Lemma~\ref{lem:localization}, applied to
$X\setminus Z$, proves the assertion.
\end{proof}

Condition~(b) of Proposition~\ref{prop:removal} concerns the
inverse image of $Z$ on the whole regularity locus $R$. We verify this
condition for cubic hypersurfaces below and for quartic double spaces
in Section~\ref{sec:double}.

\begin{lemma}\label{lem:cubicfibers}
Let $X\subset\PP^{n+1}$ be a smooth cubic hypersurface, $n\ge2$, and
let $i(p,z)$ be the third point of intersection with the line through
$p$ and $z$. Then
\[
 \dim i^{-1}(y)\le n\qquad\text{for every }y\in X.
\]
\end{lemma}

\begin{proof}
On the dense open subset of pairs of distinct points $p,z\in X$
whose joining line meets $X$ in three distinct points, $i_p(z)$ is the
third intersection point. Thus $p,z,i_p(z)$ are collinear there.
On the regularity locus $R$, the map
\[
 R\longrightarrow(\PP^{n+1})^3,
 \qquad(p,z)\longmapsto(p,z,i_p(z))
\]
is a morphism. The locus of collinear triples is closed, and its inverse
image contains the dense open subset just described. Therefore
$p,z,i_p(z)$ are collinear for every $(p,z)\in R$.

Fix $y\in X$. The intersection of $i^{-1}(y)$ with the
diagonal $\{p=z\}$ has dimension at most $n$. It remains to estimate
the part of $i^{-1}(y)$ where $p\ne z$.

For a pair with $i_p(z)=y$, the line
$\ell=\langle p,z\rangle$ passes through $y$. Lines through $y$ form
$\PP^n$. If $\ell\not\subset X$, the choices of $p,z\in X\cap\ell$
are finite, so such pairs have dimension at most $n$.

The family of lines through $y$ contained in $X$ 
is automatically contained in 
$X\cap T_yX$ which is of dimension $n-1$. 
Thus, the dimension of the family is at most $n-2$.
For each such line $\ell$, the pairs in
$i^{-1}(y)$ whose joining line is $\ell$ form a subset of
$\ell\times\ell$, of dimension at most two. Since the family of
these lines has dimension at most $n-2$, their total contribution
has dimension at most $(n-2)+2=n$.
\end{proof}

\section{Moving involutions on quartic double spaces}\label{sec:double}

Consider a smooth quartic double $n$-fold
\[
 \pi:X\longrightarrow\PP^n,\qquad n\ge2,
\]
with a smooth quartic hypersurface $S$ in $\PP^n$ 
as the branch locus.
In this section we construct a moving family 
of birational involutions 
on $X$ which satisfies the assumptions of Proposition 
\ref{prop:criterion}. By virtue of this proposition, the existence 
of such a family ensures that $X$ is elliptic. 

Let $\sigma:X\to X$ be the deck involution. 
Denote by $T\subset\Gr(1,n)$ the open subset 
of lines in $\PP^n$ meeting $S$ in four distinct points. 
Consider the  incidence relation
\[
 \EE=\{(\ell,x)\in T\times X:\pi(x)\in\ell\}.
\]
The first projection $\EE \longrightarrow T$ 
yields a smooth proper family of irreducible curves. 
By the Riemann--Hurwitz formula, 
its fiber $E_\ell=\pi^*(\ell)$ over $\ell$ is 
a smooth curve of genus one. 
We write $H_\ell$ for the degree-two divisor class on 
$E_\ell$ pulled 
back from $\OO_\ell(1)$. For every $p\in E_\ell$,
\begin{equation}\label{eq:deck}
 p+\sigma(p)\sim H_\ell.
\end{equation}

\begin{lemma}\label{lem:reflection}
There is a rational map $i:X\times X\dashrightarrow X$ 
such that, whenever $(p,x)\in X\times X$ is such that 
$\pi(p)\ne\pi(x)$ and their joining line 
$\ell$ belongs to $T$, the point $i_p(x)$ 
is the unique point of $E_\ell$ satisfying
\begin{equation}\label{eq:reflection}
 x+i_p(x)\sim2\sigma(p).
\end{equation}
There is an open subset $\Omega\subset X\times X$ such that
\[
 (p,i_p(x))\in\Omega,\qquad i_p(i_p(x))=x
 \qquad\text{for every }(p,x)\in\Omega.
\]
\end{lemma}

\begin{proof}
A degree-one divisor class on a smooth
\revisionb{curve of genus one} has a unique effective representative, 
so \eqref{eq:reflection} determines a unique point $i_p(x)$. 
To see that it varies algebraically, 
consider the morphism
\begin{equation}\label{eq:scheme}
 {\rm pr_1}\colon\EE\times_T\EE\longrightarrow\EE.
\end{equation}
The fiber of ${\rm pr_1}$ over a point $(\ell,p)\in\EE$
is the genus-one curve $E_\ell$. Since $\EE\to T$ is a smooth
proper family of genus-one curves, so is \eqref{eq:scheme}.
The latter family admits two sections
\[(\ell, p)\mapsto\left((\ell,p), (\ell,p)\right)\quad\text{and}
\quad (\ell, p)\mapsto\left((\ell, p), (\ell, \sigma(p))\right).\]
Taking the first section as origin makes
\eqref{eq:scheme} an elliptic scheme over $\EE$; the origin on the fiber
over $(\ell,p)$ is $p$. The group law on this fiber is defined by the
Abel--Jacobi isomorphism
\[
 \alpha_p:E_\ell\xrightarrow{\sim}\Pic^0(E_\ell),
 \qquad u\longmapsto[u-p],
\]
where $[D]$ denotes the linear equivalence class of a divisor $D$.
Subtracting $2p$ from both sides of \eqref{eq:reflection} gives
\[
 [x-p]+[i_p(x)-p]=2[\sigma(p)-p]
 \qquad\text{in }\Pic^0(E_\ell).
\]
All three divisor classes have degree zero. Transporting this equality
through $\alpha_p$ gives
\[
 x\longmapsto i_p(x)=2\sigma(p)-x,
\]
where doubling and subtraction are operations in the group law with
origin $p$. Thus the displayed equality is an equality of points of
$E_\ell$. In the elliptic scheme \eqref{eq:scheme}, this operation is
translation by twice the second section, composed with inversion.
Both operations are morphisms over $\EE$. This proves that $i_p(x)$
depends regularly on the line $\ell$, the center $p$, and the moving
point $x$.

Let
\[
 U_0=\{(p,x)\in X\times X:\pi(p)\ne\pi(x),\quad
       \ell=\langle\pi(p),\pi(x)\rangle\in T\}.
\]
The assignment $(p,x)\mapsto((\ell,p),(\ell,x))$ defines a morphism
$U_0\to\EE\times_T\EE$. 
Applying the reflection in each fiber sends
$((\ell,p),(\ell,x))$ to $((\ell,p),(\ell,i_p(x)))$.
The natural map $\EE\to X$, $(\ell,y)\mapsto y$, applied to the second
factor then gives the morphism $i:U_0\to X$, $(p,x)\mapsto i_p(x)$,
satisfying \eqref{eq:reflection}.

Set
\begin{equation}\label{eq:omega}
 \Omega=\{(p,x)\in U_0:\pi(i_p(x))\ne\pi(p)\}.
\end{equation}
For a pair in $\Omega$, the joining line of $\pi(p)$ 
and $\pi(i_p(x))$ is still $\ell$. Reflection on $E_\ell$ 
gives $i_p(i_p(x))=x$. Since $\pi(x)\ne\pi(p)$, 
the pair $(p,i_p(x))$ again lies in $\Omega$. Thus the map
$(p,x)\mapsto(p,i_p(x))$ is a morphism of $\Omega$ 
whose square is the identity.
\end{proof}

The next lemma is the special geometric input 
that makes Proposition~\ref{prop:criterion} 
work for quartic double spaces.

\begin{lemma}\label{lem:evaluation}
For every $x\in X$, the set
\[
 P_x=\{p\in X:(p,x)\in\Omega\}
\]
is dense open, and
\[
 e_x:P_x\longrightarrow X,\qquad p\longmapsto i_p(x)
\]
is dominant of degree four.
\end{lemma}

\begin{proof}
Set $q=\pi(x)$ and let $\mathcal T_q\subset T$ 
be the subset of lines passing through $q$ 
and meeting the branch divisor $S\subset \PP^n$ i
n four distinct points. The set $\mathcal T_q$ is a nonempty 
open subset of the projective space of lines through $q$. 
Indeed, if $q\notin S$, projection from $q$ is generically 
finite of degree four on $S$, so a general line through $q$ 
meets $S$ in four simple points. If $q\in S$, projection from 
$q$ is generically finite of degree three on $S\setminus\{q\}$; 
a general line through $q$ has simple intersections away from $q$, 
and omitting directions in the tangent hyperplane $T_qS$ makes 
the intersection at $q$ simple as well. 

We restrict the incidence family to $\mathcal T_q$ letting
\[
 \EE_q=\{(\ell,z):\ell\in\mathcal T_q,\ \pi(z)\in\ell\}
 \longrightarrow\mathcal T_q.
\]
The point $x$ gives a section of this smooth proper family
of genus-one curves. For the evaluation calculation we take this fixed
section as origin: on every fiber the zero is $x$, while the center $p$
varies. This differs from the choice of origin $p$ in the proof of
Lemma~\ref{lem:reflection}. We now use the identification
\[
 \alpha_x:E_\ell\xrightarrow{\sim}\Pic^0(E_\ell),
 \qquad u\longmapsto[u-x].
\]
Put $h_\ell=\sigma(x)\in E_\ell$. By \eqref{eq:deck} at $x$,
\[
 \alpha_x(h_\ell)=[\sigma(x)-x]=[H_\ell-2x].
\]
Thus $h_\ell$ is a point of $E_\ell$; its image under $\alpha_x$ is
the degree-zero divisor class $[H_\ell-2x]$.
Subtracting $2x$ from \eqref{eq:deck} at $p$ gives
\[
 [p-x]+[\sigma(p)-x]=[H_\ell-2x]=\alpha_x(h_\ell).
\]
Hence, in the group law with origin $x$,
\[
 \sigma(p)=h_\ell-p.
\]
Likewise, the divisor relation $x+e_x(p)\sim2\sigma(p)$ from
\eqref{eq:reflection} gives
\[
 [e_x(p)-x]=2[\sigma(p)-x],
 \qquad\text{that is,}\qquad
 \alpha_x(e_x(p))=2\alpha_x(\sigma(p)).
\]
Transporting this equality to $E_\ell$, we obtain
\begin{equation}\label{eq:minus2}
 e_x(p)=2\sigma(p)=2h_\ell-2p.
\end{equation}
Here and in the rest of this proof, sums and integer multiples of
points refer to the group law with origin $x$.

The points $h_\ell=\sigma(x)$ form the section
$h:\ell\mapsto(\ell,\sigma(x))$ of $\EE_q\to\mathcal T_q$.
Consequently the morphism
\[
 m_x:\EE_q\longrightarrow\EE_q,
 \qquad (\ell,p)\longmapsto(\ell,2h_\ell-2p),
\]
is translation by the section $2h$ composed with multiplication by $-2$
on this elliptic scheme. Multiplication by $-2$ is finite \'etale of
degree four in characteristic zero, and translation is an automorphism.
Thus $m_x$ is finite \'etale of degree four.

The morphism $r:\EE_q\to X$ forgetting $\ell$ is birational. Indeed, 
away from the point or points over $q$, a point $x'\in X$ determines 
the unique line joining $q$ to $\pi(x')$. On a dense open set, therefore,
\[
 e_x=r\circ m_x\circ r^{-1}.
\]
Thus $e_x$ extends to a dominant rational self-map of $X$ of degree four.

It remains to check that its restriction to the domain $P_x$ is still dominant. 
On a fiber $E_\ell$, the excluded centers are the points above $q$, 
together with the points $p$ satisfying
\[
 \pi(e_x(p))=\pi(p),
\]
which is equivalent to $e_x(p)=p$ or $e_x(p)=\sigma(p)$. 
By \eqref{eq:minus2}, these equations are respectively
\[
 3p=2h_\ell\qquad\text{or}\qquad p=h_\ell.
\]
They have respectively nine and one solutions on 
a smooth elliptic fiber $E_\ell$. Hence the forbidden 
locus is a proper closed subset of $\EE_q$, and its 
complement has dense birational image in $X$ contained in $P_x$. 
Therefore $P_x$ is dense, and evaluation on $P_x$ 
remains dominant of degree four.
\end{proof}

\begin{remark}\label{rem:ramification}
If $x$ lies on the ramification divisor, then $\sigma(x)=x$.
Thus $h_\ell=x$, the zero of the group law used in
Lemma~\ref{lem:evaluation}, for every $\ell\in\mathcal T_q$.
Formula~\eqref{eq:minus2} reduces to
\[
 p\longmapsto[-2]p.
\]
Thus ramification points are not exceptional for the dominance argument.
\end{remark}

For the reader's convenience, we provide the proof of 
unirationality of $X$ following Beauville \cite[Proposition~5]{Bea}. 
That paper was withdrawn because of an error in its 
stable-irrationality argument; the unirationality construction 
below is independent of it.

\begin{lemma}\label{lem:unirational}
Every smooth quartic double $n$-fold, $n\ge2$, is unirational.
\end{lemma}

\begin{proof}
For $n=2$, $X$ is a del Pezzo surface of degree two, 
hence rational over $k$. Suppose $n\ge3$. 
On the incidence family $\EE$, define $f(\ell,p)\in E_\ell$ 
to be the unique point satisfying
\[
 f(\ell,p)\sim3p-H_\ell.
\]
As in Lemma~\ref{lem:reflection}, the map $f:\EE\to X$ is regular. 
Fiberwise it is multiplication by $3$ followed by translation, 
so $f:\EE\to X$ is dominant with general fiber of dimension $n-1$.

Let us restrict $f$ to the dense open subset where 
$\pi(f(\ell,p))\ne\pi(p)$. On a fiber over $x'\in X$, 
forgetting $\ell$ is an isomorphism onto 
a locally closed subset of $X$: its inverse assigns 
the line joining $\pi(p)$ and $\pi(x')$. For general $x'$, 
the closure of this subset has dimension $n-1$, and its 
boundary has dimension at most $n-2$. 
For a general plane $P\subset\PP^n$, finiteness of $\pi$ 
implies that $\pi^{-1}(P)$ cuts the closure in dimension 
one and the boundary in dimension at most zero. 
Thus the intersection meets the original fiber. 
The incidence correspondence of triples $(P,\ell,p)$ with $\pi(p)\in P$ 
in this open subset maps dominantly to $\Gr(2,n)\times X$ by 
$(P,\ell,p)\mapsto(P,f(\ell,p))$. Its constructible image contains 
a dense open subset, whose fibers over general $P$ are dense in $X$. 
Hence, for general $P$, the restriction of $f$ 
to centers in $Y=\pi^{-1}(P)$ is dominant.

By Bertini, $Y$ is a smooth degree-two del Pezzo surface, 
hence rational. The incidence variety of lines through points 
of $Y$ is a $\PP^{n-1}$-bundle over $Y$, hence rational. 
Its open subset with transverse lines therefore dominates $X$.
\end{proof}

\begin{proof}[Proof of Theorem~\ref{thm:main}]
The variety $X$ is smooth and projective, and 
Lemma~\ref{lem:unirational} proves its unirationality. 
Lemmas~\ref{lem:reflection} and~\ref{lem:evaluation} 
verify the remaining hypotheses of Proposition~\ref{prop:criterion}. 
The conclusion follows.
\end{proof}

\begin{lemma}\label{lem:doublefibers}
For the reflection map $i:X\times X\dashrightarrow X$ of
Lemma~\ref{lem:reflection},
\[
 \dim i^{-1}(y)\le n\qquad\text{for every }y\in X.
\]
\end{lemma}

\begin{proof}
Fix $y\in X$ and put $q=\pi(y)$. The locus $\pi(p)=\pi(z)$ is
$X\times_{\PP^n}X$, of dimension $n$ because $\pi$ is finite.
It remains to estimate the part of $i^{-1}(y)$ where
$\pi(p)\ne\pi(z)$. Such a pair determines the line
$\ell=\langle\pi(p),\pi(z)\rangle$.

On the open subset $U_0$ of Lemma~\ref{lem:reflection},
the points $p,z,i_p(z)$ belong to the same curve $E_\ell$, so their
images under $\pi$ are collinear. On the regularity locus $R$, the map
\[
 R\longrightarrow(\PP^n)^3,
 \qquad(p,z)\longmapsto\bigl(\pi(p),\pi(z),\pi(i_p(z))\bigr)
\]
is a morphism. Its inverse image of the closed locus of collinear
triples contains the dense open subset $U_0$, and therefore equals
$R$. Consequently, if $i_p(z)=y$ and $\pi(p)\ne\pi(z)$, their joining
line contains $q=\pi(y)$. Lines through $q$ form $\PP^{n-1}$.

Recall that $T\subset\Gr(1,n)$ is the open subset of lines
meeting the branch hypersurface $S$ in four distinct points. We call
these lines transverse to $S$. For a transverse line $\ell\in T$
through $q$, the pair $(p,z)$ lies in the open subset $U_0$ of
Lemma~\ref{lem:reflection}, and \eqref{eq:reflection} gives
\[
 z+y\sim2\sigma(p)\qquad\text{on }E_\ell.
\]
For each $p\in E_\ell$, there is a unique $z\in E_\ell$ satisfying
this relation. Thus the pairs over a fixed transverse line have
dimension at most one, and varying the line gives dimension at
most $(n-1)+1=n$.

By the proof of Lemma~\ref{lem:evaluation}, the transverse lines
through $q$ form a nonempty open subset $\mathcal T_q$ of
$\PP^{n-1}$, including when $q\in S$. Its complement has dimension
at most $n-2$. For each line in this complement, the pairs
under consideration form a subset of
$\pi^{-1}(\ell)\times\pi^{-1}(\ell)$, of dimension at most two.
These pairs therefore have dimension at most $(n-2)+2=n$ as well.
This estimate includes lines contained in $S$ and does not require
a reflection formula on singular fibers.
\end{proof}

\begin{proof}[Proof of Theorem~\ref{thm:removal}]
Suppose first that $X$ is a smooth quartic double space. The
hypotheses of Proposition~\ref{prop:criterion} have been verified in
the proof of Theorem~\ref{thm:main}.

Recall that $R$ is the regularity locus of $i$, and that
\[
 \Omega_p=\{z\in X:(p,z)\in\Omega\},\qquad
 R_p=\{z\in X:(p,z)\in R\}.
\]
Fix $p\in X$ and a transverse line $\ell$ through $\pi(p)$. Let
$r_{p,\ell}:E_\ell\to E_\ell$ be the reflection of
Lemma~\ref{lem:reflection}. By the definition of $\Omega$, a point
$z\in E_\ell$ belongs to $\Omega_p$ precisely when
\[
 \pi(z)\ne\pi(p),\qquad
 \pi(r_{p,\ell}(z))\ne\pi(p).
\]
The excluded points $z\in E_\ell$ are therefore exactly those in
\[
 \{p,\sigma(p)\}\ \cup\
 r_{p,\ell}^{-1}\bigl(\{p,\sigma(p)\}\bigr).
\]
This set is finite, since $r_{p,\ell}$ is an automorphism of $E_\ell$.
Thus $\Omega_p$ meets each such curve in a dense open subset.
The curves above transverse lines through $\pi(p)$ cover a dense
open subset of $X$, so $\Omega_p$ is dense in $X$.
Since $\Omega_p\subset R_p$ and $i_p:\Omega_p\to\Omega_p$ is an
involution, $R_p$ is dense and $i_p:R_p\to X$ is dominant.
This verifies condition~(a) of Proposition~\ref{prop:removal},
including when $p$ lies on the ramification divisor.

For a smooth cubic hypersurface, unirationality and the
moving-involution hypotheses are the classical inputs used in
\cite{KZcubic}. For every $p\in X$, projection from $p$ is
generically of degree two, and $i_p$ is its birational involution.
The third-intersection map is regular on a dense open subset of
$\{p\}\times X$ and agrees there with $i_p$. Thus $R_p$ is dense
and its restriction is dominant.

For either family, Lemma~\ref{lem:doublefibers} or
Lemma~\ref{lem:cubicfibers} verifies condition~(b) of
Proposition~\ref{prop:removal}:
\[
 \dim i^{-1}(Z)\le n+\dim Z\le2n-2.
\]
Proposition~\ref{prop:removal} now applies.
\end{proof}

\section{Ellipticity of punctured cones}\label{sec:cones}

We first give a sufficient condition for the complement of the zero section
of a line bundle to be algebraically elliptic.

\begin{lemma}\label{lem:bundlelift}
Let $X$ be a smooth variety and $M$ a line bundle on $X$. 
Suppose that for every $x\in X$ there are an open neighborhood 
$U$ of $x$, finitely many morphisms
\[
 S_i:U\times\PP^1\longrightarrow X,\qquad S_i(z,a_i)=z,
\]
and marked points $a_i\in\PP^1$ such that
\[
 \sum_i(dS_i)_{(x,a_i)}(0,T_{a_i}\PP^1)=T_xX,
 \qquad \deg S_i(x,\cdot)^*M<0\quad\text{for every }i.
\]
Then $M^\times$, the complement in ${\rm tot}(M)$ 
of the zero section, is algebraically elliptic.
\end{lemma}

\begin{proof}
Write $Y=M^\times$ and let $\rho:Y\to X$ be the bundle projection. 
Fix $x\in X$ and choose $U, S_i$ and $a_i$ as before. Shrink $U$ 
so that it is connected and $M|_U$ is trivial.

The projective-bundle formula for the Picard group gives
\[
 \Pic(U\times\PP^1)\simeq\Pic(U)\oplus\mathbb Z,
\]
where the second summand is generated by 
$\pr_{\PP^1}^*\OO_{\PP^1}(1)$; see \cite[Remark~1.8(b)]{Isch}. 
In particular,
\[
 d_i=-\deg S_i(z,\cdot)^*M
\]
is independent of $z\in U$ and is positive. 
Consider the induced  line bundle $S_i^*M$ on $U\times\PP^1$. 
Its restriction to the marked section $U\times\{a_i\}$ 
can be identified with $M|_U$. After choosing a basis 
of $\OO_{\PP^1}(-d_i)_{a_i}$, we can therefore choose an isomorphism
\[
 S_i^*M\simeq\pr_U^*(M|_U)\otimes
                   \pr_{\PP^1}^*\OO_{\PP^1}(-d_i)
\]
whose restriction to $U\times\{a_i\}$ is the identity on $M|_U$.

Fix a trivialization of $M|_U$. Removing the zero sections 
identifies $(S_i^*M)^\times$ with $U\times Q_i$, where
\[
 Q_i=\OO_{\PP^1}(-d_i)^\times,\qquad q_i:Q_i\longrightarrow\PP^1.
\]
The natural map from the pullback bundle to $M$ induces a morphism
\[
 F_i:U\times Q_i\longrightarrow Y,\qquad
 \rho(F_i(z,\xi))=S_i(z,q_i(\xi)).
\]
Our normalization along the marked section identifies $Y|_U$ 
with $U\times q_i^{-1}(a_i)$. Let $\alpha_i:Y|_U\hookrightarrow U\times Q_i$ 
denote the resulting inclusion. Then
\[
 F_i\circ \alpha_i=\id_{Y|_U}.
\]
For each fixed $z\in U$, the map $F_i$ thus identifies 
$\{z\}\times q_i^{-1}(a_i)$ with the fiber $\rho^{-1}(z)$.

We now construct sprays on $Q_i$. The isomorphism
$\OO_{\PP^1}(-1)^\times\simeq\A^2\setminus\{0\}$ 
sends a nonzero vector $v$ to $([v],v)$. 
Taking the $d_i$th tensor power in each fiber gives
\[
 ([v],v)\longmapsto([v],v^{\otimes d_i}),\qquad
 Q_i\simeq(\A^2\setminus\{0\})/\mu_{d_i}.
\]
Here $\mu_{d_i}$ acts by scalar multiplication. 
It commutes with $\mathrm{SL}_2$, whose transitive 
action on $\A^2\setminus\{0\}$ therefore descends to 
a transitive action on $Q_i$. 
Since $\mathrm{SL}_2$ is generated by its 
$\Ga$-subgroups, $Q_i$ is flexible in terminology 
of \cite[Proposition~1.1]{AKZ12}.  
Thus, by \cite[Proposition 1.8]{AFKKZ} one can find 
$\Ga$-subgroups $u_j:\Ga\to\mathrm{SL}_2, \, j=1, \ldots, l$ such that 
for every $\xi\in Q_i$ 
the derivatives of  $t\mapsto u_j(t)\cdot\xi$ at $t=0$ span $T_\xi Q_i$.

For $y\in Y|_U$, write $\alpha_i(y)=(z,\xi_i(y))$, 
with $z=\rho(y)$, and define
\[
 \widetilde s_{ij}:Y|_U\times\A^1\longrightarrow Y,\qquad
 \widetilde s_{ij}(y,t)=F_i\bigl(z,u_j(t)\cdot\xi_i(y)\bigr).
\]
These are local rank-one sprays with values in $Y$, 
since $\widetilde s_{ij}(y,0)=y$. Their images need not lie in $Y|_U$.

Let $y\in\rho^{-1}(x)$. For each $i$, the parameter derivatives of 
these sprays indexed by $j$ span
\[
 V_i=(dF_i)_{\alpha_i(y)}\bigl(0,T_{\xi_i(y)}Q_i\bigr)\subset T_yY.
\]
The subspace $V_i$ contains $\ker(d\rho)_y$, 
because $F_i$ identifies the fiber over $a_i$ with $\rho^{-1}(x)$. 
Moreover, differentiating the identity for $\rho\circ F_i$ and 
using the surjectivity of $(dq_i)_{\xi_i(y)}$ gives
\[
 (d\rho)_y(V_i)=(dS_i)_{(x,a_i)}(0,T_{a_i}\PP^1).
\]
These images span $T_xX$ by hypothesis. The exact sequence
\[
 0\longrightarrow\ker(d\rho)_y\longrightarrow T_yY
 \xrightarrow{(d\rho)_y}T_xX\longrightarrow0
\]
therefore implies $\sum_i V_i=T_yY$. 
The sprays $\widetilde s_{ij}$ jointly dominate at every point 
of $\rho^{-1}(x)$. Since $x$ was arbitrary, 
Lemma~\ref{lem:localization} proves ellipticity of $Y$.
\end{proof}

\begin{remark}\label{rem:degree}
The nonzero-degree condition is essential to this construction. For $d_i=0$,
the surface $Q_i=\PP^1\times(\A^1)^\times$, where
$(\A^1)^\times=\A^1\setminus\{0\}$, admits no dominating algebraic spray:
every morphism from an affine space to $(\A^1)^\times$ is constant.
Although $M|_U$ is trivial, the pullback $S_i^*M$ retains the factor
$\pr_{\PP^1}^*\OO_{\PP^1}(-d_i)$, which is nontrivial on
$U\times\PP^1$, since $d_i>0$.
\end{remark}

We now apply Lemma~\ref{lem:bundlelift} to the complete rational curves
constructed in the proof of Proposition~\ref{prop:criterion}.

\begin{theorem}\label{thm:generalcone}
Let $X$ be a smooth complete unirational variety of dimension $n\ge2$
satisfying the moving-involution hypotheses of Proposition~\ref{prop:criterion}.
Then $X$ is algebraically elliptic. If $L$ is an ample line bundle on $X$,
then $(L^{-1})^\times$ is algebraically elliptic.
\end{theorem}

\begin{proof}
Proposition~\ref{prop:criterion} proves ellipticity of $X$. 
Its proof also supplies, near every $x\in X$, morphisms
\[
 S_i:U\times\PP^1\longrightarrow X,\qquad S_i(z,a_i)=z,
\]
whose parameter derivatives at $(x,a_i)$ span $T_xX$: 
the sprays in~\eqref{eq:spray} extend to the complete 
parameter lines $L_i$, since $\Psi$ is regular on $W\times L_i$. 
Each map $S_i(x,\cdot)$ is nonconstant. If $L$ is ample, 
then $X$ is projective, since it is complete. Put $M=L^{-1}$. We have
\[
 \deg S_i(x,\cdot)^*M=-\deg S_i(x,\cdot)^*L<0.
\]
Lemma~\ref{lem:bundlelift} proves ellipticity of $M^\times$, 
which is the punctured affine cone associated with $L$.
\end{proof}

\begin{proof}[Proof of Corollary~\ref{cor:cone}]
Theorem~\ref{thm:generalcone} proves statement~(a). 
Statement~(b) follows from \cite[Theorem~1.2]{Kus}. 
Over $\mathbb C$, statement~(c) follows from \cite[Theorem~1.6]{For}. 
Finally, $Y$ is quasiaffine and elliptic, 
so \cite[Proposition~6.1]{KZcone} gives~(d).
\end{proof}

\begin{remark}
Theorem~\ref{thm:generalcone} applies to smooth 
cubic hypersurfaces as well as to smooth quartic double spaces. 
In the cubic case, the third-intersection involutions 
satisfy the assumptions of Proposition~\ref{prop:criterion}, 
and unirationality is classical. 
Thus the theorem gives a new proof of the ellipticity assertion 
in \cite[Corollary~1.4]{KZcubic} for cones defined by ample line bundles. 
The earlier proof invoked \cite[Corollary~3.8]{KZcone} without 
establishing the required curve-orbit property for cubics. 
The present proof supplies the missing argument for these cones. 
Generalized cones defined by ample $\mathbb Q$-divisors 
are not included in Theorem~\ref{thm:generalcone}.
\end{remark}

\begin{corollary}\label{cor:nontrivialbundle} Let $X$ be 
a smooth quartic double space of dimension $n\ge 2$.
Then for every nontrivial line bundle $M$ on $X$, $M^\times$ 
is algebraically elliptic. 
\end{corollary}
\begin{proof}
For $ n \ge 3$, $\Pic(X) = \mathbb Z[\pi^*\OO_{\PP^n}(1)]$, 
so $M$ or $M^{-1}$ is ample, and 
$M^\times \simeq (M^{-1})^\times$. The assertion therefore follows from Theorem~\ref{thm:generalcone}. For $n = 2$, $X$ 
is a rational del Pezzo surface of degree 2,
and the assertion follows from \cite[Theorem~5.11]{Kal}. 
\end{proof}

\medskip

\centerline{AI disclosure}

\medskip

During the development of this project, 
the authors used GPT-6.0 Astra 
as a research assistant. 
After being directed to some previous work by the authors,
specific examples and mathematical techniques, the model
assisted in developing some of the arguments 
presented in this paper. 
The authors independently
checked the mathematical arguments, rewrote the proofs, 
and developed the consequences presented
here. AI tools were also used for grammatical corrections 
and improvements of the exposition. The
authors take full responsibility for the content and 
correctness of the paper.

\end{document}